\documentclass[12pt,twoside,reqno]{amsart}

\usepackage[OT1]{fontenc}
\RequirePackage[english]{babel}

\usepackage{enumerate}

\usepackage{amsthm}
\RequirePackage{amsmath,amsfonts,amssymb,amsthm}

\RequirePackage[dvips]{graphicx}

\usepackage{psfrag}
\usepackage{verbatim}

\usepackage[monochrome]{color}

\usepackage[dvips]{graphicx} 

  \numberwithin{equation}{section}

  \newcommand{\N}{\mathbb{N}}         
  \newcommand{\R}{\mathbb{R}}         
  \newcommand{\T}{\mathbb{T}}		  

  \newcommand{\PP}{\mathbb{P}}        
  \newcommand{\supp}{\text{supp}}        
  \newcommand{\diam}{\text{diam}}       

  \newtheorem{theorem}{Theorem}[section]
  \newtheorem{lemma}[theorem]{Lemma}
  \newtheorem{prop}[theorem]{Proposition}

\newtheorem{remark}[theorem]{Remark}
\newtheorem{remarks}[theorem]{Remarks}

  \theoremstyle{remark}
    \newtheorem*{defn}{Definition}

  \newtheorem{ex}[theorem]{Example}

  \DeclareMathOperator{\ldimloc}{\underline{dim}_{loc}}
    
   \DeclareMathOperator{\lad}{\underline{A}}
      \DeclareMathOperator{\uad}{\overline{A}}
         \DeclareMathOperator{\ad}{A}
           
            \DeclareMathOperator{\esssup}{ess sup}    
\DeclareSymbolFont{bbold}{U}{bbold}{m}{n}
\DeclareSymbolFontAlphabet{\mathbbold}{bbold}


\begin{document}

\title{Random cutout sets with spatially inhomogeneous intensities}

\author{Tuomo Ojala}
\address{University of Jyvaskyla, Department of Mathematics and Statistics, P.O. Box 
35 (MaD), FI-40014 University of Jyvaskyla, Finland}
\email{tuomo.ojala@gmail.com}

\author{Ville Suomala}
\address{Department of Mathematical Sciences, University of Oulu, Finland}
\email{ville.suomala@oulu.fi}

\author{Meng Wu}
\address{Department of Mathematical Sciences, University of Oulu, Finland}
\email{meng.wu@oulu.fi}

\begin{abstract}
We study the Hausdorff dimension of Poissonian cutout sets defined via inhomogeneous intensity measures on $Q$-regular metric spaces. {\color{blue} Our main results explain the {\color{red}dependence} of the dimension of the cutout sets {\color{red}on} the multifractal structure of the average densities of the $Q$-regular measure.} As a corollary, we obtain formulas for the Hausdorff dimension of such cutout  sets in self-similar and self-conformal spaces using the multifractal decomposition of the average densities for the natural measures.
\end{abstract}

\thanks{V.S. and M.W. acknowledge financial support from the Academy of Finland.}

\maketitle

\section{Introduction}

Given a metric space $X$ and a sequence of open balls $B(x_n,r_n)\subset X$, we define the \emph{cutout set} corresponding to the sequence $(x_n,r_n)\in X\times(0,1)$, {\color{blue} $n\in\N$,} as
\begin{align*}
E&=X\setminus\bigcup_n B(x_n,r_n)\,.
\end{align*}
That is, $E$ is the set left uncovered by the union of the balls $B(x_n,r_n)$. {\color{blue}{If the cutout balls $B(x_n,r_n)$ are randomly distributed, or if their centers are dynamically defined (e.g. if $x_{n+1}=T(x_n)$ for a given dynamics $T\colon X\rightarrow X$)}, it is of interest to investigate whether $E\neq\varnothing$ and to determine its structure and size such as Hausdorff dimension. 

For random cutouts, this problem originates from different versions of the Dvoretzky covering problem as well as in the study of renewal sets (see e.g. {\color{red}\cite{Dvoretzky,Kahane1959,Mandelbrot72,Shepp1972}}). Kahane's book \cite{Kahane85} as well as his survey \cite{Kahane00} on random coverings provide a detailed account of the history of the problem. 
In turn, the dynamical version of the problem arises from {\color{red} dynamical} Diophantine approximation {\color{red}\cite{FanSchmeling03, Fanetal13, LiaoSeuret13}}. See also \cite{JonassonSteif08} for a variant dealing with time-dependent dynamics.}

In this paper, we consider only the case in which {\color{blue}$(x_n,r_n)_{n\in\N}$} are random variables. We shall next describe our model in detail.
Let $X=(X,\mathcal{H},d)$ be a bounded metric space endowed with a 
measure $\mathcal{H}$, which is (Ahlfors-David) $Q$-regular for some $0<Q<\infty$: there are constants $0<c_0<C_0<\infty$ such that
\begin{equation}\label{eq:Q-reg}
c_0r^Q\le\mathcal{H}\left(B(x,r)\right)\le C_0 r^Q\,,
\end{equation}
for all $x\in X$, $0<r<\diam(X)$. Throughout the paper, a measure will refer to a locally finite Borel regular outer measure.

 For each $0<\gamma<+\infty$, let $\mathcal{Y}$ be a Poisson point process on $X\times(0,1)$ with intensity $\gamma\mathcal{H}\times\rho$, where $\rho$ is the measure defined by $\rho(dr)=\tfrac{dr}{r^{Q+1}}$ on $(0,1)$. Thus, $\mathcal{Y}$ is a random collection of pairs $(x,r)\in X\times(0,1)$ such that 
\begin{enumerate}
\item For each Borel set $\mathcal{A}\subset X\times(0,1)$, the random variable $\sharp(\mathcal{A}\cap\mathcal{Y})$ is Poisson distributed with mean $\gamma\mathcal{H}\times\rho(\mathcal{A})$.
\item For disjoint $\mathcal{A}_i$, the random variables $\sharp(\mathcal{A}_i\cap\mathcal{Y})$ are independent. 
\end{enumerate}
In particular, $\mathcal{Y}$ is {\color{blue} almost surely } countably infinite. We consider the random cutout set:
\begin{align}\label{eq definition of Poisson cutout}
E=X\setminus\bigcup_{(x,r)\in\mathcal{Y}}B(x,r)\,.
\end{align}
Note that the intensity of $\mathcal{Y}$ and the induced probability $\mathbb{P}$ crucially depend on $\gamma$. Let
\begin{equation}\label{eq:gamma0}
\gamma_0:=\sup\{\gamma>0\,:\mathbb{P}(E\neq\varnothing)>0\}\,.
\end{equation}
A central problem is to determine the exact value of $\gamma_0$ ($0<\gamma_0<\infty$ always holds, see Remark \ref{rem:coarse}). Further, when $0<\gamma<\gamma_0$, we would like to determine the almost sure Hausdorff dimension of $E$. Since for any $\gamma>0$, there is a positive probability for extinction ($E=\varnothing$), {\color{blue} we will consider the following quantity:
\begin{defn}
Given a Poisson point process $\mathcal{Y}$ as above, the associated \emph{cutout dimension} is the unique value $s\ge 0$ such that $\dim_H(E)\le s$ almost surely and for all $t<s$, there is a positive probability that $\dim_H (E)>t$.
\end{defn}
We note that Z\"ahle \cite{Zahle85} uses the term \emph{essential dimension of $E$} for the cutout dimension.}

The case when $X$ is a {\color{blue} bounded } subdomain of some Euclidean space $\R^d$ and $\mathcal{H}=\mathcal{L}^d$ is the $d$-dimensional Lebesgue measure is well understood. In particular, $\gamma_0=\gamma_0(d)=d/ \alpha(d)$, where $\alpha(d)=\mathcal{L}^d(B(0,1))$ and for $0<\gamma\le d/\alpha(d)$, {\color{blue} the cutout dimension (and this holds also if Hausdorff dimension is replaced by box-dimension in the definition)} equals 
$d-\gamma \alpha(d)$, see \cite{ElHelou78, Zahle85, Thacker06, NacuWerner11, ShmerkinSuomala}. 
In this case, the point process $\mathcal{Y}$ is translation invariant in an obvious way, but it possesses also strong scale invariance: If $I,\lambda I\subset(0,1)$ for {\color{blue}some} $\lambda>0$, then it is equally likely that a point $x\in X$ is covered by a ball $B(x_n,r_n)\in\mathcal{Y}$ for $r_n\in I$ as it is for $r_n\in\lambda I$. There are many works (e.g. \cite{Mandelbrot72, Fetal85, Zahle85, Rivero03, BiermeEstrade12}) in which this {\color{red} scale invariance condition} has been relaxed by replacing the measure $\rho(dr)$ by a more general measure of the form $\tfrac{dr}{h(r)}$. For such  generalizations, it is still possible to get results on the size of $E$ and the range of $\gamma$, for which $E\neq\varnothing$ with positive probability. However, it turns out that the model is much more sensitive for the changes in the spatial component $\mathcal{H}$ and in essentially all the works we are aware of, only the case in which $\mathcal{H}=\mathcal{L}^d$ has been considered. The papers \cite{HJ73} and \cite{Zahle85} are notable  exceptions. In these papers, various estimates for the dimension of the cutout sets are obtained in the context of a general metric space. However, when it comes to determining the exact value of the cutout dimension, 
{\color{red} it is assumed  that $\mathcal{H}=\mathcal{L}^d$ in \cite{Zahle85}.} {\color{blue} Also in \cite{HJ73}, a strong homogeneity assumption on $\mathcal{H}$ (implying in particular that $\sup_{x,y\in X}\tfrac{\mathcal{H}(B(x,r))}{\mathcal{H}(B(y,r))}\longrightarrow 1$ as $r\downarrow 0$) is assumed.} 

We note that in many of the references given above, the model is actually one where $r_n$ is a deterministic sequence and $x_n$ are independent and uniformly distributed {\color{red} according to {\color{blue} $\tfrac{\mathcal{L}^d|_X}{\mathcal{L}^d(X)}$ (and often $X$ is the torus $\mathbb{T}^d$)}}. However, in the case of translation invariant intensity the methods in the case of deterministic radii and iid centres are essentially the same as in the Poissonian case described above ({\color{blue}e.g. in \cite[Section 10]{Kahane1990} it is explained how to reduce the case of deterministic radii to a Poissonian case}). 

Furthermore, in many of the cited works, a significant part of attention has been given to the study of the \emph{random covering set} 
\[F=\bigcap_{k\in\N}\bigcup_{n\ge k}B(x_n,r_n)\,,\]
consisting of the points covered infinitely often by the balls $B(x_n,r_n)$. However, under the present assumptions and for any choice of $\gamma$, it follows from Fubini's theorem that {\color{blue} almost surely } $\mathcal{H}(X\setminus F)=0$ so that {\color{red}$F$ has the same dimension as $X$}. { \color{blue} Furthermore}, for the case of deterministic radii, as well as for more general Poissonian intensities $\mathcal{H}(dx)\times\tfrac{dr}{h(r)}$, the dimensional properties of the associated random covering set in the setting of $Q$-regular spaces are analogous to the Euclidean situation (where $\mathcal{H}$ is the Lebesgue measure). 
For instance, the proof of \cite[Proposition 4.7]{Jetal14} adapts easily to the case of $Q$-regular metric spaces. These observations indicate that changing the spatial component of the intensity measure does not affect the 
{\color{red} determination of the dimensions of}
 the random covering sets, as opposed to the {\color{red} same problem for} the cutout set $E$. 

Before going further, let us provide a simple example to get an idea why the lack of homogeneity in $\mathcal{H}$ is a subtle issue for the cutouts. {\color{red} Here, by homogeneity we mean that $\mathcal{H}(B(x,r))=\mathcal{H}(B(y,r))$ for any $x,y\in X$} {\color{blue} and small enough $r>0$}. Suppose $X=X_1\cup X_2$ where, say, $X_1$ and $X_2$ are disjoint {\color{blue} open} subintervals of $[0,1]$. Let $\mu=a\mathcal{L}_{|X_1}+b\mathcal{L}_{|X_2}$ and suppose $0<a<b<\tfrac12$. Now, conditional on $E\cap X_1=\varnothing$, we know from the above discussion that {\color{blue} almost surely } $\dim_H (E)\le 1-2b$ while on $E\cap X_1\neq\varnothing$, there is a positive probability that $\dim_H (E)=1-2a$. This shows that one cannot expect any a.s. constancy result for the Hausdorff dimension of $E$. Of course, it still holds that the cutout dimension is $1-2a$ (see also Remark \ref{rems:final}).

{\color{blue}
 For each $0<t<1$, let \[E_t=X\setminus\bigcup_{(x,r)\in\mathcal{Y}\,,r>t}B(x,r)\] 
 and for $x\in X$ denote $p(x,t)=\mathbb{P}(x\in E_t)$. Notice that $x\in E_t$ if and only if $\mathcal{A}_t(x)\bigcap \mathcal{Y}=\varnothing$ for 
\[\mathcal{A}_t(x)=\{(y,r)\,:\,r>t,\,y\in B(x,r)\}\,,\]
so that
\begin{equation}\label{equation introduction 1}
\begin{split}
p(x,t)&=\mathbb{P}(\mathcal{A}_t(x)\cap\mathcal{Y}=\varnothing)=\exp\left(-\gamma\mathcal{H}\times\rho(\mathcal{A}_t(x))\right)\\
&=\exp\left(-\gamma\int_{r=t}^{1}\mathcal{H}(B(x,r))r^{-Q-1}\,dr\right)\,.
\end{split}
\end{equation}
}

{\color{red}
{\color{blue} We may rewrite the identity \eqref{equation introduction 1} as 
\begin{equation}\label{Dprel}
 p(x,t)=t^{\gamma \ad(\mathcal{H},x,t)}\,
\end{equation}
where 
\begin{equation}\label{eq:ad_def}
\ad(\mathcal{H},x,t)=\frac{\int_{r=t}^1\mathcal{H}(B(x,r))r^{-Q-1}\,dr}{-\log t}\,.
\end{equation}
{\color{blue} In particular, the expected measure of $E_t$ equals
\begin{equation}\label{eq:expectedarea}
\mathbb{E}(\mathcal{H}(E_t))=\int_{x\in X}p(x,t)\,d\mathcal{H}(x)=\int_{x\in X}t^{\gamma A(\mathcal{H},x,t)}\,d\mathcal{H}(x)\,.
\end{equation}
} 
These formulas} suggest an intimate connection to the lower and upper ($Q$-)average densities of $\mathcal{H}$ 
defined at $x\in X$ as
\begin{align}
\lad(\mathcal{H},x)&=\liminf_{t\rightarrow 0}\ad(\mathcal{H},x,t)\,,\\
\uad(\mathcal{H},x)&=\limsup_{t\rightarrow 0}\ad(\mathcal{H},x,t)\,.
\end{align}
If $\lad(\mathcal{H},x)$ and $\uad(\mathcal{H},x)$ coincide, we denote the common value by $\ad(\mathcal{H},x)$.} 
It is well known that for a {\color{red}$Q$-regular measure $\mathcal{H}$}, the density $\lim_{r\downarrow 0}\tfrac{\mathcal{H}(B(x,r))}{r^Q}$ fails to exist at $\mathcal{H}$-almost all points, {\color{blue} unless $Q$ is an integer and $X$ is rectifiable (see \cite{Preiss87})}. However,  for many important $Q$-regular measures (see e.g. \cite{BedfordFisher92, Falconer92, Falconer97, Zahle98, Zahle01}), the average density $\ad(\mathcal{H},x)$ is known to exist and take a constant value $\alpha$ at $\mathcal{H}$-almost all points of $X=\supp(\mathcal{H})$. Recalling \eqref{eq:expectedarea}, a first naive guess would be to predict that in such a case the cutout dimension would equal $Q-\gamma\alpha$. However, it turns out that in most cases of interest, the dimension of $E$ is affected by the zero measure set, where $\ad(\mathcal{H},x)\neq\alpha$ and a finer analysis of the multifractal properties of the average densities is needed in order to catch the correct dimension of the cutout set $E$.

{\color{red}In the following, we present our main results. For this, we need some more notations and definitions. 
For {\color{blue} $0\le\alpha<\beta\le\infty$ } and $0<r<1$, we define
{\color{blue}
\begin{equation}\label{introduction equation x(a,b,r)}
\Delta^{\mathcal{H}}_r(\alpha,\beta)=\{x\in X\mid\alpha<\ad(\mathcal{H},x,r)<\beta\}
\end{equation}
and 
\begin{equation}\label{introduction equation x(a)}
\Delta^{\mathcal{H}}(\alpha)=\{x\in X\mid \ad(\mathcal{H},x)=\alpha\}.
\end{equation}
In what follows, the measure $\mathcal{H}$ will always be clear from the context, and thus we may ignore it in the notation and write $\Delta_r(\alpha,\beta)$ and $\Delta(\alpha)$ instead of the more rigorous $\Delta^{\mathcal{H}}_r(\alpha,\beta)$ and $\Delta^{\mathcal{H}}(\alpha)$. 
Let $\alpha_{\min}=\inf\{\alpha : \Delta(\alpha)\neq \varnothing\}$, $\alpha_{\max}=\sup\{\alpha : \Delta(\alpha)\neq \varnothing\}$ and $\alpha_0={\rm esssup}_\mathcal{H} \uad(\mathcal{H},x)$.  
Finally, let
\begin{equation}\label{introduction equation spectrum}
f(\alpha)=\dim_H (\Delta(\alpha))
\end{equation}
} 
and
\begin{equation}\label{introduction equation formula dim of cutout}
{\color{blue}m(\gamma)=\sup_{\alpha\ge 0}f(\alpha)-\gamma \alpha\,.}
\end{equation}

The following theorem is our main general result which says that if $f(\alpha)$ is continuous on $]\alpha_{\min},\alpha_{\max}[$ and the quantity $A(\mathcal{H},x ,r)$ satisfies a large deviation principle then the cutout dimension is given by $m(\gamma)$. {\color{blue} Recall that $0<\gamma<\infty$ is the parameter used to adjust the intensity of $\mathcal{Y}$.} 

\begin{theorem}\label{main theorem}
Suppose that $f(\alpha)$ is continuous on $]\alpha_{\min},\alpha_{\max}[$ {\color{cyan} and for all $0<\beta\le\alpha_0$ and all $\varepsilon>0$, there is $0<C<\infty$ such that
\begin{equation}\label{LD condition}
\mathcal H (\Delta_r(0,\beta))\le Cr^{Q-f(\beta)-\varepsilon}\,,
\end{equation}
whenever $r>0$.}
If $m(\gamma)\ge 0$, then almost surely $\dim_H(E)\le m(\gamma)$ and $\dim_H(E)=m(\gamma)$ with positive probability. If $m(\gamma)<0$, then $E=\varnothing$ almost surely.
\end{theorem}

In Section \ref{sec:applications}, Theorem \ref{main theorem} is applied in the case when $X$ is a $C^{1+\varepsilon}$ self-conformal set satisfying the strong separation condition and the spatial component $\mathcal{H}$ for the Poisson intensity is the natural self-conformal measure {\color{blue}on} $X$. We show that in this case, $(X,\mathcal{H})$ satisfies the hypothesis of Theorem \ref{main theorem}, thus the cutout dimension {\color{cyan} equals} $m(\gamma)$.
}

The structure of the paper is as follows. In Section \ref{sec:aux}, using the familiar first and second moment methods, we present some tools to estimate the dimension of the intersections of $E$ with certain sub- and superlevel sets of the average densities $\uad(\mathcal{H},\cdot)$, $\lad(\mathcal{H},\cdot)$. This part applies to any $Q$-regular measure and can be used directly to obtain some (coarse) estimates on the value of $\gamma_0$ {\color{blue}(recall \eqref{eq:gamma0})} and on the dimension of $E$. {\color{blue} Combining these tools and the assumption on the multifractal spectrum $f(\alpha)$ allows us to  give a proof for Theorem \ref{main theorem}, this is provided in the end of the Section.}

{\color{red}In Section \ref{sec:applications}, we present our second main result, Theorem \ref{thm:main},}
{\color{blue} which is an application of Theorem \ref{main theorem}}; We consider the case when $X$ is {\color{blue} a self-conformal set  satisfying the strong separation condition and $\mathcal{H}$ is the natural self-conformal measure, which is well known to be $Q$-regular with $Q=\dim_H(X)$.} Using tools from thermodynamical formalism and expressing the average densities {\color{blue} of $\mathcal{H}$ } as ergodic averages, we examine the multifractal spectrum {\color{blue} $f(\alpha)$ for $\mathcal{H}$. Theorem \ref{main theorem} then  } enables us to obtain a formula for $\gamma_0$ and for the cutout dimension when $0<\gamma<\gamma_0$. {\color{blue} The condition \eqref{LD condition} for self-conformal spaces is verified in the Appendix following Section \ref{sec:applications}.} 

\section{Auxiliary dimension estimates}\label{sec:aux}



In this section, we provide some useful upper and lower estimates for the Hausdorff dimension of $E\cap\{\alpha<\ad(\mathcal{H},x)<\beta\}$ when $\alpha$ and $\beta$ vary. Our standing assumption is that $\mathcal{H}$ is a $Q$-regular measure on the metric space $X$. Further, the parameter $\gamma>0$ that determines (together with $\mathcal{H}$) the intensity of $\mathcal{Y}$ is fixed throughout the section. {\color{blue} For $Y\subset X$ and $t>0$, we denote by
$Y(t)=\{y\in X\,:\,d(y,Y)\le t\}$,
the (closed) $t$-neighbourhood of $Y$.} 

\subsection{Dimension upper bound}

Recall that for each $0\le\alpha<\beta\le\infty$ and $0<r<1$, we denote {\color{blue} $\Delta_r(\alpha,\beta)=\{x\in X\mid\alpha<\ad(\mathcal{H},x,r)<\beta\}$, where $\ad(\mathcal{H},x,r)$ is as in \eqref{eq:ad_def}.}

\begin{lemma}\label{lemma:neighbourhood} 

{\rm (i)} There exists $C<\infty$, independent of {\color{cyan}  $x$ and }$t$, such that 
$$\PP(x\in E(t))\leq C\PP(x\in E_t).$$

{\rm(ii)} If $0 \leq \alpha' < \alpha < \beta < \beta' < \infty$, there exists $r_0>0$ such that 
$$ \Delta_r(\alpha,\beta)(r) \subset \Delta_r(\alpha',\beta') $$
for all $0<r<r_0$.
\end{lemma}

\begin{proof}
 (i)   Observe that by definition of $E(t)$ and elementary geometry, we have 
$$E(t)\subset X \setminus \bigcup_{(x,r)\in\mathcal{Y},r>t}B(x,r-t) =: E_t'\,.$$
So $\PP(x\in E(t))\leq \PP (x\in E_t')$. Thus, we only need to show that  $\PP (x\in E_t')\leq C\PP(x\in E_t)$ for some $C< \infty$ independent of {\color{cyan} $x$ and }$t$. Since $x\in E'_t$ if and only if $\mathcal{A}\cap\mathcal{Y}=\varnothing$ for 
\[\mathcal{A}=\{(y,r)\,:\,r>t,\,y\in B(x,r-t)\}\,,\]
we deduce that
\begin{equation}\label{equation lemma:neighbourhood 1}
\PP(x\in E_t'){\color{blue}=\PP(\mathcal{A}\cap\mathcal{Y}=\varnothing)=} \exp\left(-\gamma\int_{t}^1\mathcal{H}(B(x,r-t))\frac{dr}{r^{Q+1}}\right)\,.
\end{equation}
Now, we have 
\begin{equation}\label{equation lemma:neighbourhood 2}
\begin{split}
\int_{t}^1\mathcal{H}(B(x,r-t))\frac{dr}{r^{Q+1}}&\geq \int_{2t}^1\mathcal{H}(B(x,r-t))\frac{dr}{r^{Q+1}}\\
&=\int_{t}^1\mathcal{H}(B(x,r))\frac{dr}{r^{Q+1}(1+t/r)^{Q+1}}\,.
\end{split}
\end{equation}
An elementary calculation shows that
$$\frac{1}{(1+y)^{Q+1}}\geq 1-(Q+1)y \ \ \textrm{for all }\ y\in [0,1].$$
Applying this to $y=t/r$ in (\ref{equation lemma:neighbourhood 2}), we get 
$$
\int_{t}^1\mathcal{H}(B(x,r-t))\frac{dr}{r^{Q+1}} \geq   \int_{t}^1\mathcal{H}(B(x,r))\frac{dr}{r^{Q+1}}-(Q+1)t\int_{t}^1\mathcal{H}(B(x,r))\frac{dr}{r^{Q+2}} $$
Since  $C'=\sup_{{\color{cyan} x\in X, 0<t\leq 1 }}(Q+1)t\int_{t}^1\mathcal{H}(B(x,r))\frac{dr}{r^{Q+2}}<+\infty$, substituting the above inequality in (\ref{equation lemma:neighbourhood 1})  yields
$$\PP(x\in E_t')\leq \PP(x\in E_t)\exp(\gamma C').$$
Letting $C=\exp(\gamma C')$ ends the proof of (i).

(ii) We have seen in the proof of (i) that there exists $C<+\infty$, such that 
$$\int_{t}^1\mathcal{H}(B(x,r)\setminus B(x,r-t))\frac{dr}{r^{Q+1}} \leq C.$$
By the same argument, it follows that
$$\int_{t}^1\mathcal{H}(B(x,r+t)\setminus B(x,r-t))\frac{dr}{r^{Q+1}} \leq C'$$
for some $C'<+\infty$ independent of {\color{cyan} $x$ and }$t$.
Thus for every $\varepsilon >0$ there exists $r_0 >0$ such that 
\[{\color{blue}\frac{\int_{t}^1 \mathcal{H}(B(x,r+t) \setminus B(x,r-t))r^{-Q-1}\,dr}{-\log t} \leq \varepsilon}\]
for every {\color{cyan} $x\in X$, }$0 < t < r_0$.
  Since for every $x \in \Delta_t(\alpha,\beta)(t)$, there exists $y\in \Delta_t(\alpha,\beta)$ such that $d(x,y) < t$,  we deduce that
{\color{blue}
  \[ A(\mathcal{H},x,t) \leq A(\mathcal{H},y,t) + \frac{\int_{t}^1 \mathcal{H}(B(x,r+t) \setminus B(x,r-t))r^{-Q-1}\,dr}{-\log t} < \beta' \,,\]} 
  when $\varepsilon < \beta' - \beta$.
  The lower bound follows by a similar calculation. 
\end{proof}

\begin{lemma}\label{lemma:up}
Let $0<\alpha'<\alpha<\beta<\beta'\le\infty$ and $C,\eta\ge 0$. Suppose that $\mathcal H(\Delta_r(\alpha',\beta'))\le Cr^\eta$ for all $0<r<1$. Then {\color{blue} almost surely, }
\[\dim_H\left(E\cap \limsup_{r\downarrow 0}\Delta_r(\alpha,\beta)\right)\le Q-\gamma\alpha'-\eta\,,\]
if $Q-\gamma\alpha'-\eta\ge 0$ while $E\cap \limsup_{r\downarrow 0}\Delta_r(\alpha,\beta)=\varnothing$ if $Q-\gamma\alpha'-\eta<0$.
\end{lemma}

\begin{proof}
Observe that by \eqref{Dprel}, $\mathbb P(x\in E_r)\le r^{\gamma\alpha'}$, for $x\in \Delta_r(\alpha',\beta')$. Pick $\alpha'<\widetilde{\alpha}<\alpha$, $\beta<\widetilde{\beta}<\beta'$. Using Lemma \ref{lemma:neighbourhood}, we have for each $\theta<\gamma\alpha'+\eta$ that
\begin{align*}
&\mathbb{E}\left(\sum_{n\in\N}2^{\theta n}\mathcal{H}\left( \left(\Delta_{2^{-n}}(\widetilde{\alpha},\widetilde{\beta})\cap E\right)(2^{-n})\right)\right)\\
\le&C_1\sum_{n}2^{\theta n}\int_{\Delta_{2^{-n}}(\alpha',\beta')}\mathbb{P}\left(x\in E_{2^{-n}}\right)\,d\mathcal{H}(x)\\
\le& C_1\sum_n 2^{n\theta}\mathcal H(\Delta_{2^{-n}}(\alpha',\beta'))2^{-n\gamma\alpha'}\\
\le& C_2\sum_{n}2^{n(\theta-\gamma\alpha'-\eta)}<\infty\,.
\end{align*}
In particular,
we see that almost surely, 
\[\lim_{n\to\infty}2^{\theta n}\mathcal{H}\left(\left( \Delta_{2^{-n}}(\widetilde{\alpha},\widetilde{\beta})\cap E \right)(2^{-n})\right)=0\,.\]
Since $\mathcal{H}$ is $Q$-regular this implies {\color{blue} almost surely} the existence of $N_0\in\N$ such that for all $n\ge N_0$, the set $(\Delta_{2^{-n}}(\widetilde{\alpha},\widetilde{\beta})\cap E)$ is covered by a union of balls 
$$B\left(x_{n,1},2^{-n}\right),\ldots, B\left(x_{n,m_n},2^{-n}\right)$$ with $m_n\le 2^{n(Q-\theta)}$. Since 
\[\limsup_{r\downarrow 0}\Delta_r(\alpha,\beta)\subset\bigcup_{n=N}^\infty\bigcup_{i=1}^{m_n}B\left(x_{n,i},2^{-n}\right)\,,\] 
for all $N\ge N_0$, and 
\begin{align*}
\sum_{n\ge N}m_n 2^{-n(Q-\theta+\varepsilon)}\le \sum_{n\ge N}2^{-n\varepsilon}\longrightarrow 0\,,
\end{align*} 
for any $\varepsilon>0$, this implies the claim. Note that if {\color{blue}$Q-\gamma\alpha'-\eta<0$}, we have $m_n=0$ and thus $(\Delta_{2^{-n}}(\widetilde{\alpha},\widetilde{\beta})\cap E)(2^{-n})=\varnothing$ for all $n\ge N_0$.
\end{proof}

\subsection{A lower estimate}

Let $\mu$ be a measure on $X$. For each $t>0$, we define a measure $\nu_t$ by
\begin{equation}\label{nut}
d\nu_t(x)=p(x,t)^{-1}\mathbf{1}_{E_t}(x)\,d\mu(x)\,.
\end{equation}
{\color{blue} Recall that $p(x,t)=\PP(x\in E_t)$.} Then $(\nu_t)_{t>0}$ is a $T$-martingale in the sense of Kahane \cite{Kahane87} and it is easy to check that {\color{blue} almost surely,} $\nu_t$ is weakly convergent to a random limit measure $\nu$. 

Let $0<s<\infty$ be such that
\begin{equation}\label{eq:energy}
\int_X\int_X d(x,y)^{-s}\,d\mu(x)\,d\mu(y)<\infty\,,
\end{equation}
and define a Kernel $K\colon X\times X\to[0,\infty[$ by 
\begin{equation}\label{eq:kernel}
K(x,y)=d(x,y)^{-s}{\color{blue}p\left(x,d(x,y)\right)\,.}
\end{equation}

\begin{lemma}\label{lemma:energy}
\[\mathbb{E}\left(\int\int K(x,y)\,d\nu(x)\,d\nu(y)\right)<\infty\,.\]
\end{lemma}

\begin{proof}
It suffices to show that for all $0<t<1$,
\begin{equation}\label{eq:Efinitet}
\mathbb{E}\left(\int\int K(x,y)\,d\nu_t(x)\,d\nu_t(y)\right)<C<\infty\,,
\end{equation}
where $C$ is independent of $t$. Indeed, using that $x\mapsto \ad(\mathcal{H},x,r)$ is continuous (this follows e.g. from the calculation in the proof of Lemma \ref{lemma:neighbourhood}) and recalling \eqref{Dprel} allows to express $K(x,y)$ as a limit of increasing continuous functions, so that \eqref{eq:Efinitet} yields the claim.

We first claim that for all $0<\delta<1$,
\begin{equation}\label{eq:joint_probab}
\mathbb{P}(x,y\in E_\delta)\le C
{\color{blue}p(x,\delta)p(y,\delta)/p\left(x,d(x,y)\right)\,,}
\end{equation}
where $C$ is independent of $\delta$ and $d(x,y)$.
Indeed, this is a result of direct calculation (we assume that $\delta<d(x,y)/2$ as otherwise \eqref{eq:joint_probab} follows directly from \eqref{Dprel}): {\color{blue}Letting $\mathcal{A}=\{(z,r)\,:\,r>\delta, \{x,y\}\cap B(z,r)\neq\varnothing\}$, it follows that $x,y\in E_\delta$ precisely when $\mathcal{A}\cap\mathcal{Y}=\varnothing$. Moreover}
\begin{align*}
  {\color{blue}\mathcal{H}\times\rho(\mathcal{A})=}\int_{\delta}^1& \mathcal{H}(B(x,s) \cup B(y,s)) s^{-Q-1} ds\\
  \ge& \int_{\delta}^1 \mathcal{H}(B(y,s)) s^{-Q-1} ds + \int_{\delta}^{d(x,y)/2} \mathcal{H}( B(x,s)) s^{-Q-1} ds \\
  \ge& \int_{\delta}^1 \mathcal{H}(B(y,s)) s^{-Q-1} ds + \int_{\delta}^1 \mathcal{H}(B(x,s)) s^{-Q-1} ds \\
  &- \int_{d(x,y)}^{1} \mathcal{H}( B(x,s)) s^{-Q-1} ds - C_1\,,\\
\end{align*}
where $C_1$ is a constant such that $\int_{d(x,y)/2}^{d(x,y)} \mathcal{H}( B(x,s)) s^{-Q-1} ds \leq C_1$ and thus only depends on the $Q$-regularity data of the measure $\mathcal{H}$.
The claim \eqref{eq:joint_probab} now follows by multiplying the inequality by $-\gamma$ and taking the exponential {\color{blue}(recall \eqref{equation introduction 1})}.

Combining \eqref{eq:joint_probab}, Fubini's theorem, and \eqref{eq:energy} we calculate
\begin{align*}
&\mathbb{E}\left(\int\int K(x,y)\,d\nu_t(x)\,d\nu_t(y)\right)\\
=&\int_X\int_X\frac{\mathbb{P}(x,y\in E_t)p(x,d(x,y)))d(x,y)^{-s}}{p(x,t) p(y,t)}\,d\mu(x)\,d\mu(y)\\
&\le C\int\int d(x,y)^{-s}d\mu(x)d\mu(y)<\infty\,.
\end{align*}
Since this upper bound is independent of $t$, we are done.
\end{proof}

The following lemma employs the standard connection between capacity and dimension in the situation at hand. Recall that the lower local dimension of a measure $\nu$ at $x\in X$ is defined as
\[\ldimloc(\nu,x)=\liminf_{r\downarrow0}\frac{\log\nu(B(x,r))}{\log r}\,.\]

\begin{lemma}\label{lemma:lb}
{\color{blue} Let $\nu_t$ and $\nu$ be defined via \eqref{nut} and suppose that \eqref{eq:energy} holds with} $s-\gamma\alpha>0$. If for $\mu$-almost all $x\in X$, $\uad(\mathcal H,x)<\alpha$, then $\nu(X)>0$ with positive probability and
almost surely,
\[\ldimloc(\nu,x)\ge s-\gamma\alpha\,,\]
for $\nu$-almost all $x\in X$.
\end{lemma}

\begin{proof}
We first observe that if $N\subset X$ is $\mu$-null, then it is almost surely $\nu$-null. Indeed, for each $\varepsilon>0$, there is an open set $U_\varepsilon\supset N$, such that $\mu(U_\varepsilon)<\varepsilon$. Thus Fatou's lemma gives
\begin{align*}
\mathbb E(\nu(N))\le \mathbb E(\nu(U_\varepsilon))\le\mathbb E(\liminf_{t\downarrow 0}\nu_t(U_\varepsilon))\le\liminf_{t\downarrow 0}\mathbb{E}\nu_t(U_\varepsilon)=\mu(U_\varepsilon)<\varepsilon\,.
\end{align*}
Whence $\mathbb E(\nu(N))=0$, or in other words, $\nu(N)=0$ almost surely.

Let 
\[F_M=\{x\in X\mid \uad(\mathcal H,x)<\alpha\text{ and }\int_{y\in X}K(x,y)\,d\nu<M\}\,.\] 
Then, by the above and Lemma \ref{lemma:energy}, it follows that {\color{blue} almost surely} 
\begin{equation*}
\nu(X\setminus F_M)\longrightarrow 0\text{ as }M\longrightarrow\infty\,. 
\end{equation*}
On the other hand, for all $x\in F_M$, and all small enough $0<r<1$, \eqref{eq:kernel} and \eqref{Dprel} give
$K(x,y)\ge d(x,y)^{\gamma\alpha-s}\ge r^{\gamma\alpha-s}$ for $y\in B(x,r)$ and whence
\[r^{\gamma\alpha-s}\nu(B(x,r))\le\int_{y\in B(x,r)} K(x,y)\,d\nu<M\,,\]
implying
$\nu(B(x,r))\le M r^{s-\gamma\alpha}$. The second claim of the Lemma now follows by taking logarithms, letting $r\downarrow 0$ and finally letting $M\longrightarrow\infty$.

To prove that $\nu(X)>0$ is an event of positive probability, we first pick so small $r_0>0$ that $\mu(F)>0$, where $F=\{x\in X\mid\ad(\mathcal H,x,r)<\alpha\text{ for all }0<r<r_0\}$. Calculating as in the proof of Lemma \ref{lemma:energy} yields
\begin{align*}
&\mathbb{E}\left(\nu_t(F)^2\right)\le C
\int_{x\in F}\int_{y\in F}d(x,y)^{-\gamma\ad(\mathcal H,x,d(x,y))}\,d\mu(x)\,d\mu(y)\\
&\le C\int\int d(x,y)^{-s}d\mu(x)d\mu(y)<\infty\,.
\end{align*}
In other words, $\nu_t(F)$ is an $L^2$-bounded martingale with nonzero expectation (since $\mu(F)>0$). Whence, $\nu(X)\ge \nu(F)>0$ with positive probability.
\end{proof}

\begin{remarks}\label{rem:coarse}
(i) Lemmas \ref{lemma:up} and \ref{lemma:lb} can be used directly to obtain upper and lower estimates on $\gamma_0$ and on the dimension of $E$. Let {\color{cyan} $d_0=\inf_{x\in X}\uad(\mathcal{H},x)$, $D_0=\mu-\esssup_{x\in X}\uad(\mathcal{H},x)$} (note that $c_0\le d_0\le D_0\le C_0$,{\color{blue} where $c_0,C_0$ are as in \eqref{eq:Q-reg}}). Applying Lemma \ref{lemma:up} with $\eta=0$, implies $\gamma_0\le Q/d_0$ and $\dim_H(E)\le Q-\gamma d_0$ a.s, if $0<\gamma\le\gamma_0$. In turn, Lemma \ref{lemma:lb} applied for $\mu=\mathcal{H}$ and $s=Q$, gives the estimate $\gamma_0\ge Q/D_0$ and provided $0<\gamma<Q/D_0$, implies that  $\dim_H(E)\ge Q-\gamma D_0$ with positive probability. {\color{cyan} Although we will not deal with packing dimension later in this paper, we mention that if $d_1=\inf_{x\in X}\lad(\mathcal{H},x)$, then a modification of Lemma \ref{lemma:up}, where $\limsup \Delta_r(\alpha,\beta)$ is replaced by $\liminf\Delta_r(\alpha,\beta)$ implies that the packing dimension of $E$ is at most $Q-\gamma d_1$ ($0<\gamma\le\gamma_0$).}

(ii) {\color{blue} As indicated by Theorem \ref{main theorem}, } even if $\lad(\mathcal{H},x)=\uad(\mathcal{H},x)=\alpha$ for $\mathcal{H}$-almost every $x$, such direct estimates are usually far from being sharp. Actually, as will be seen in the Section \ref{sec:applications}, the dimension of $E$ depends intimately on the multifractal properties of the average density of $\mathcal{H}$.

(iii) As of curiosity, we mention that if $X=\T^d$ is the $d$-dimensional torus {\color{blue} (or any open subset of $\R^d$)} and $Q=d$, then $\gamma_0=d/d_0$ and for $\gamma\le\gamma_0$, the cutout dimension is $d-\gamma d_0$. Indeed, for each  $c>d_0$, there is a point $x\in X$ and $r>0$ such that $\mathcal{H}(B(x,r))\le c r^d$. An application of the Lebesgue density theorem yields a Borel set $B\subset B(x,r)$ such that $\mathcal{H}(B)>0$ and $\uad(\mathcal{H},x)<c$ for all $x\in B$. Lemma \ref{lemma:lb} {\color{blue} applied to $\mu=\mathcal{H}|_B$ then implies } that $\dim_H(E)\ge Q-\gamma c$ is an event of positive probability.
\end{remarks}

\subsection{Proof of Theorem \ref{main theorem}}

{\color{red} Now, we are ready for the proof of Theorem \ref{main theorem}. Recall the definitions of  $\Delta(\alpha), \alpha_{\min}, \alpha_{\max},\alpha_0, f(\alpha) $ and $m(\gamma)$ from the Introduction, see \eqref{introduction equation x(a)}-\eqref{introduction equation formula dim of cutout}.}

\begin{proof}[Proof of Theorem \ref{main theorem}]
Suppose that $m(\gamma)\ge0$.
We first consider the upper bound. Since trivially $\mathbb E(\mathcal H(\Delta_r(\alpha_0,+\infty))\le\mathcal H(X)=C<\infty$, {\color{blue}using that $\{x\,|\,\uad(\mathcal{H},x)\ge\alpha_0\}\subset\limsup_{r\downarrow0}\Delta_r(\alpha_0-\varepsilon,+\infty)$ for all $\varepsilon>0$, Lemma \ref{lemma:up} implies that almost surely,
\begin{equation}\label{gea0}
\dim_H\left(E\cap\{x\mid \uad(\mathcal H,x)\ge\alpha_0\}\right)\le \max\{0,Q-\gamma\alpha_0\}=\max\{0, f(\alpha_0)-\gamma\alpha_0\}\le m(\gamma)\,.
\end{equation}
Note that \eqref{LD condition} implies in particular, that $f(\alpha_0)=Q$.

To deal with the points where $\uad(\mathcal{H},x)<\alpha_0$, we first remark that Lemma \ref{lemma:up} together with \eqref{LD condition} imply $\uad(\mathcal{H},x)\ge\alpha_{\min}$ for all $x\in X$. Indeed, if $\alpha<\alpha_{\min}$, then for $\varepsilon>0$ small enough, \eqref{LD condition} gives $\mathcal{H}(\Delta_r(\alpha-\varepsilon,\alpha+\varepsilon))=O(r^{Q-\varepsilon})$, and applying Lemma \ref{lemma:up} with $\eta=Q-\varepsilon$ implies $\{x\in E\,|\,\uad(\mathcal{H},x)=\alpha\}\subset E\cap\limsup_{r\downarrow0}\Delta_r(\alpha-\varepsilon,\alpha+\varepsilon)=\varnothing$.} 

Next, let {\color{blue}$\alpha_{\min}\le\alpha<\beta<\alpha_0<\alpha_{\max}$ (the case $\alpha_0=\alpha_{\max}$ reduces to the estimate \eqref{gea0}).} Combining Lemma \ref{lemma:up} and \eqref{LD condition} {\color{blue} and using that $\{\alpha\le\uad(\mathcal{H},x)\le\beta\}\subset\limsup_{r\downarrow 0}\Delta_r(\alpha-\varepsilon,\beta+\varepsilon)$}, gives for all small $\varepsilon>0$ that
\[\dim_H\left(E\cap\{x\mid \alpha\le \uad(\mathcal H,x))\le\beta\}\right)\le \max\{0,f(\beta+\varepsilon)-\gamma(\alpha-\varepsilon)+\varepsilon\}\,.\]
Letting $\varepsilon\downarrow 0$ and using the continuity of $f$ on $]\alpha_{\min},\alpha_{\max}[$
{\color{blue} then implies
\begin{align*}
&\dim_H\left(E\cap\{x\mid \uad(\mathcal H,x))<\alpha_{0}\}\right)\\
&\le \max_{0\le k\le n-1}\{0,f\left(\alpha_{n,k+1}\right)-\gamma\alpha_{n,k+1}\}+\frac{\gamma(\alpha_0-\alpha_{\min})}{n}\,,
\end{align*}
where for each $n\in\N$, $\alpha_{\min}=\alpha_{n,0}<\alpha_{n,1}\le\ldots<\alpha_{n,n}=\alpha_0$ are equally spaced points on $[\alpha_{\min},\alpha_{0}]$.} Letting $n\rightarrow\infty$ and using the continuity of $f$ on $]\alpha_{\min},\alpha_{\max}[$ once more, finally yields the almost sure upper bound 
\begin{equation*}
\dim_H\left(E\cap\{x\mid \uad(\mathcal H,x))<\alpha_0\}\right)\le m(\gamma)\,.
\end{equation*}
{\color{blue} Combining with \eqref{gea0} we have that almost surely $\dim_H(E)\le m(\gamma)$, as required.}

If $m(\gamma)<0$, a straightforward modification of the argument using the latter claim of Lemma \ref{lemma:up} implies $E=\varnothing$ almost surely.

To prove the lower bound, let $\varepsilon>0$ and pick $\alpha$ such that 
\[{\color{blue}m(\gamma)+\varepsilon>\sup_{\alpha\ge0}f(\alpha)-\gamma \alpha>0}\] 
{\color{cyan} By Frostman's lemma,  there exists a probability measure $\mu_\alpha$ supported on $X$ such that $\mu_\alpha(\Delta(\alpha))=1$ and further
\begin{equation*}
\int\int d(x,y)^{\varepsilon-f(\alpha)}d\mu_\alpha(x)d\mu_\alpha(y)<C<\infty\,.
\end{equation*}
}
Consider $\nu_t$ as in \eqref{nut} and $\nu$ such that $\nu_t\rightharpoonup\nu$.
Lemma \ref{lemma:lb} implies that with positive probability $\nu(X)>0$ and further (applying the lemma with  {\color{blue}$f(\alpha)+\varepsilon$} and letting $\varepsilon\downarrow 0$) {\color{blue} almost surely}
\[\ldimloc(\nu,x)\ge m(\gamma)\,,\]
for $\nu$-almost all $x\in X$. Since $\supp(\nu)\subset E$, this shows in particular that $\dim_H(E)\ge m(\gamma)$ with positive probability.
\end{proof}

\begin{remark}
The method presented {\color{blue} in this section} works for more general gauge functions $h\colon(0,1)\rightarrow(0,+\infty)$ and measures $\mathcal{H}$ so that $C^{-1}<\mathcal{H}(B(x,r))/h(r)<C$ for some $C<\infty$. In this case the Poissonian intensity is $\gamma\mathcal{H}(dx)\times\tfrac{dr}{r h(r)}$ {\color{blue} and the $h$-average densities are defined via \[\ad_h(\mathcal{H},x,r)=\frac{\int_{r=t}^1\mathcal{H}(B(x,r))(rh(r))^{-1}\,dr}{-\log t}\,.\]} 
In the above, we have considered the case $h(r)=r^Q$, for simplicity of notation and because {\color{blue} our main applications,} the self-conformal measures in Section \ref{sec:applications}, are $Q$-regular.
\end{remark}

\section{Application to self-conformal spaces}\label{sec:applications}

Let $M$ be a $d$-dimensional Riemann manifold and $G=\{g_i\}_{i=1}^\ell$ a conformal iterated function system (IFS) of class $C^{1+\varepsilon}$ on $M$, i.e., $g_i$ are conformal contractions with tangent maps satisfying a H\"older
condition of exponent $\varepsilon$. Let $X\subset M$ be the self-conformal set corresponding to $G$, that is, $X$ is the unique compact set satisfying $X=\bigcup_{i=1}^\ell g_i(X)$. We suppose that the IFS $G$ satisfies the strong separation condition {\color{red}(SSC)}, i.e., $g_i(X)\cap g_j(X)=\emptyset$ for $i\neq j$. Let $S:X\to X$ be the inverse map of $G$ on $X$, that is, the restriction of $S$ on $g_i(X)$ is $g_i^{-1}$. Then $(X,S)$ becomes a dynamical system. It is well known that (see e.g. \cite[Chapter 5]{Falconer97}) there exists a unique probability measure $\mathcal{H}$ on $X$, called the natural measure, which is $S$-invariant, ergodic and $Q$-regular, $Q$ being the Hausdorff dimension of $X$.

{\color{red} We consider the Poisson cutout set $E$ in $X$ as defined in the Introduction (see \eqref{eq definition of Poisson cutout}). Recall that the intensity of the Poisson process $\mathcal{Y}$ is $\gamma \mathcal{H}\times \rho$ where $\mathcal{H}$ is the natural measure on $X$ {\color{cyan} and $\rho(dr)=r^{-Q-1}dr$.}

We will apply our main result (Theorem \ref{main theorem}) to determine the cutout dimension in the situation at hand. 
}

Instead of considering the continuous sequence $\{\ad(\mathcal{H},x,r),r>0\}$,   we will use the discrete one $\{\ad(\mathcal{H},x,|DS^n(x)|^{-1}),n\in \N\}$, where $DS^n$ is the tangent map of $S^n$.  Since $|DS^{n+1}(x)|/|DS^n(x)|=|DS(S^n(x))|\in (1,\max_y |DS(y)|)$ for all $n\geq 1$, the limit behavior of  $\ad(H,x,r)$ when $r\to0$ is the same as that of $\ad(\mathcal{H},x,|DS^n(x)|^{-1})$ when $n\to\infty$.

We write 
$$\ad(\mathcal{H},x,|DS^n(x)|^{-1})=\frac{1}{\log|DS^n(x)|}\sum_{k=0}^{n-1}f_k(x),$$
where 
\begin{equation*}
f_k(x)=\int_{|DS^{k+1}(x)|^{-1}}^{|DS^k(x)|^{-1}}\mathcal{H}(B(x,t))t^{-Q-1}\,dt\,.
\end{equation*} 
In our context, it is known that (see e.g.  \cite[Proposition 4.1]{Falconer92}, \cite[Chapter 6.2]{Falconer97}) there exists a sequence $(\varepsilon_n)_n$ of positive reals with $  \varepsilon_n\to 0$ such that 
\begin{equation} \label{Asymptotic additive}
|f_n(S^kx)-f_{n+k}(x)|<\varepsilon_n 
\end{equation}
for all $x\in X$  and  all $k\geq 0$. 

Recall that $\Delta(\alpha)=\Delta^{\mathcal{H}}(\alpha)=\{x\in X\,:\,\ad(\mathcal{H},x)=\alpha\}$ and that
$f(\alpha)=\dim_H(\Delta(\alpha))$. 

{\color{red} To effectively apply Theorem \ref{main theorem} to the set $E$, we need to verify that the functions $\ad(\mathcal{H},x,r)$ and $f(\alpha)$ satisfy the hypothesis of Theorem \ref{main theorem}.}
We will make use of the multifractal properties of $\Delta(\alpha)$ that we present now. First, we introduce some notions and results.
For simplicity of presentation, we express these results in the context of self-conformal sets/measures, although they are valid in  much more general settings.

{\bf Notations.} Let $\Lambda=\{1,\cdots,\ell\}.$ Recall that $g_i$, for $i\in \Lambda$, are conformal contractions. For $u=u_1\cdots u_k\in \Lambda^k$ we write  $g_u=g_{u_1}\circ\cdots\circ g_{u_k}$. Let $X_u=g_u(X)$. Denote $\Lambda^*=\bigcup_{n\geq 1}\Lambda^n$ and for $u\in\Lambda^*$, let $[u]=\{(v_n)_{n\ge 1}\in\Lambda^\infty\,:\,v_1=u_1,\ldots, v_n=u_n\}$. For any $x\in X$, there exists $(u_n)_{n\geq1}\in \Lambda^\infty$ such that $\{x\}=\lim_n g_{u_1^n}(X)=:g_{u_1^\infty}(X)$ where we write $u_1^n=u_1\cdots u_n$. The transformation $S$ can be defined as $\{S(x)\}=g_{u_2^\infty}(X)$.

A sequence $\Phi=\{\varphi_n\}$ of functions $\varphi_n : X\to\R$ is called {\em asymptotically additive} if for each $\varepsilon >0$ there exists a continuous function $\varphi:X\to\R$  such that 
\begin{equation}\label{definition asymptotic additive}
\limsup_{n\to\infty}\frac1n\sup_{x\in X}|\varphi_n(x)-A_n\varphi(x)|<\varepsilon
\end{equation}
where $A_n\varphi=\sum_{k=0}^{n-1}\varphi\circ S^k$. If $\varphi_n=A_n\varphi$ for all $n$, then $\Phi$ is called {\em additive}.

As a consequence of \eqref{Asymptotic additive}, the sequence $\{\sum_{k=0}^{n-1}f_k\}_n$ is asymptotically additive.
Indeed, for any $\varepsilon>0$, there exists $N\geq1$ such that when $\varepsilon_N<\varepsilon$, then by \eqref{Asymptotic additive} we have
$$
\limsup_{n\to\infty}\frac1n\sup_{x\in X}|\sum_{k=0}^{n-1}f_k(x)-A_nf_N(x)|<\varepsilon_N.
$$

Now, we introduce the notion of pressure function.
Let $\Phi=\{\varphi_n\}_n$ be a sequence of continuous function $\varphi_n:X\to\R$. The pressure function associated to $\Phi$ is defined by
\begin{equation}\label{definition pressure}
P(\Phi)=\limsup_{n\to\infty} \frac{1}{n}\log \sum_{u\in \Lambda^n}\sup_{x\in X_u}\exp(\varphi_n(x)).
\end{equation}
Actually, when $\Phi$ is asymptotically additive, we can replace limsup by lim in the definition of $P(\Phi)$.  In fact, from the asymptotically additivity of $\varphi_n$, we deduce that for any $\varepsilon>0$ there exists $\varphi : X\to\R$ such that
\begin{equation}\label{equation asymptotic additive property1}
\sup_{x\in X}|\varphi_n(x)-A_n\varphi(x)|\le n\varepsilon, \ \ {\rm for }\ n\gg1.
\end{equation}
So, $B_n:=\sum_{v\in \Lambda^n}\sup_{x\in X_v}\exp(\varphi_n(x))=(Ce^{ n\varepsilon})^{\pm}\sum_{v\in \Lambda^n}\sup_{x\in X_v}\exp(A_n\varphi(x))$ for some constant $C>0$, where the notation $A=C^{\pm}B$ means that $C^{-1} B\le A\le C B$. Since the sequence $\tilde{B}_n:=\sum_{v\in \Lambda^n}\sup_{x\in X_v}\exp(A_n\varphi(x))$ is sub-multiplicative, 
the limit $\lim_n\frac1n\log  \tilde{B}_n$ exists. So we have $$|\liminf_n\frac1n\log  B_n-\limsup_n\frac1n\log  B_n |\leq \varepsilon.$$ 
Since  $\varepsilon$ is arbitrary, the limit $\lim_n\frac1n\log  B_n$ exists.


Let $\mathcal{M}(X,S)$ be the set of all $S$-invariant probability measures on $X$. For $\mu\in \mathcal{M}(X,S)$ and   an asymptotically additive sequence $\Phi=\{\varphi_n\}$, define $$\Phi_*(\mu):=\lim_{n\to\infty}\int_X\frac{\varphi_n(x)}{n}d\mu(x).$$
By \eqref{definition asymptotic additive}, the limit in the above definition exists. Note that since $\mu$ is $S$-invariant, we have $\int_X\frac{A_n\varphi(x)}{n}d\mu(x)=\int_X\varphi\,d\mu$ for all $n$.
(If $\mu$ is ergodic, then by Birkhoff's ergodic theorem we deduce that $\Phi_*(\mu)$ is the $\mu$-almost sure limit of
$\frac{\varphi_n(x)}{n}$ as $n\to\infty$).
 Further, it is  known (see \cite[Lemma A.4.]{FengHuang10}, \cite[Proposition 4]{BarreiraCaoWang13}) that the map $\mu\mapsto \Phi_*(\mu) $ is continuous in the weak-star topology.

Let us return to the set $\Delta(\alpha)$. Denote 
$F=\{\sum_{k=0}^{n-1}f_k\}_n$ and $\log DS=\{\log |DS^n|\}_n$. Then $F$ is asymptotically additive and $DS$ is additive. Let 
$$\Omega=\left\{\frac{F_*(\mu)}{\log DS_*(\mu)} : \mu\in \mathcal{M}(X,T)\right\}.$$

We will use the following multifractal properties (Proposition \ref{Prop multifractal properties}) of $\Delta(\alpha)$, most of them are from \cite[Theorem 1]{BarreiraCaoWang13} (see also \cite{FengHuang10}). Before presenting those properties, we need to introduce the notion of $u$-dimension. We will present this notion in our setting of self-conformal sets/measures.

Let $u:X\to \R^+$ be a continuous function. For each word $v\in \Lambda^n$, we write 
$$u(v)=\sup\left\{\sum_{k=0}^{n-1}u(S^kx) : x\in X_v\right\}.$$
Given a set $F\subset X$ and $\alpha\in \R$, we define
$$N(F,\alpha,u)=\lim_{n\to\infty}\inf_{\Gamma}\sum_{v\in \Gamma}\exp(-\alpha u(v))$$
where the infimum is taken over all countable collections $\Gamma\in \cup_{k\geq n}\Lambda^k$ such that $F\subset \cup_{v\in \Gamma}X_v$. The $u$-dimension of $F$ with respect to $S$ is defined by 
$$\dim_u(F)=\inf\left\{\alpha\in \R : N(F,\alpha,u)=0\right\}.$$
Note that if $u=\log|DS|$, 
then the $u$-dimension $\dim_u(F)$ coincides with the Hausdorff dimension $\dim_H(F)$. This follows immediately from the existence of constants $c_1,c_2>0$ such that $c_1({\rm diam}X_v )^\alpha\le \exp(-\alpha u(v))\le c_2({\rm diam}X_v )^\alpha$. 

\begin{prop}\label{Prop multifractal properties}
The following statements hold:
\begin{enumerate}
\item The set $\Omega$ is a closed interval. \label{multifractal range}

\item We have $\Delta(\alpha)\neq \emptyset$ if and only if $\alpha\in \Omega$ and if $\alpha\in\Omega$, then 
$$\dim_u(\Delta(\alpha))=\max\left\{\frac{h_\mu(S)}{\int_X u \ d\mu} : \mu\in \mathcal{M}(X,T) \  {\rm and}\ \frac{F_*(\mu)}{\log DS_*(\mu)} =\alpha\right\}.$$ 
In particular,
$$f(\alpha)=\max\left\{\frac{h_\mu(S)}{\int_X \log |DS| \ d\mu} : \mu\in \mathcal{M}(X,T) \  {\rm and}\ \frac{F_*(\mu)}{\log DS_*(\mu)} =\alpha\right\}.$$
Here, $h_\mu(S)$ denotes the measure-theoretic entropy of $\mu$ with respect to $S$.
\label{multifractal spectrum formula}

\item The function $f$ attains its maximum at some $\alpha_{min}<\alpha_0<\alpha_{max}$ and $f(\alpha_0)=Q$. \label{dimension-max}

\item $\ad(\mathcal{H},x)=\alpha_0$ for $\mathcal{H}$-almost all $x\in X$.  \label{dimension-max-measure}

\item The function  $f : {\rm int} (\Omega) \to \R$ is continuous. \label{continuity spectrum}

\item If 
$\alpha\in\Omega$, then
$$\inf_{q\in \R} P(q(F-\alpha \log DS)-f(\alpha)\log DS)=0.$$  \label{multifractal pressure zero point}

\item  {\color{cyan}For all $0<\beta\le\alpha_0$ and all $\varepsilon>0$, we have $\mathcal H (\Delta_r(0,\beta))\le C r^{Q-f(\beta)-\varepsilon}$ for all $r>0$.}\label{LD estimate}
\end{enumerate}
\end{prop}

\begin{proof}  
The statements \eqref{multifractal spectrum formula}, \eqref{continuity spectrum} and \eqref{multifractal pressure zero point} can be found in \cite[Theorem 1]{BarreiraCaoWang13}. Note that the definition of pressure function given in \cite{BarreiraCaoWang13} is different from ours, but these two definitions actually give the same pressure function (see \cite[Sections 2.2 and 4.2.2]{Barreirabook}, \cite[Proposition 3]{PesinPitskel84}).

The statements \eqref{dimension-max} and \eqref{dimension-max-measure} can be deduced from \cite{Falconer92}: in Proposition 4.1 of \cite{Falconer92} it is proved that there exists a constant $\alpha_0>0$ such that $A(\mathcal{H},x) =\alpha_0$ for $\mathcal{H}$-almost every $x$, so $f(\alpha_0)=\dim(\mathcal{H})=\dim_H(X)=Q$ which is the maximum of $f$.

For the statement \eqref{multifractal range}, since  the map $\mu\mapsto \frac{F_*(\mu)}{\log DS_*(\mu)}$ is continuous and $\mathcal{M}(X,T)$ is a compact and convex set,  we only need to notice that a subset of $\R$, which is the image of a compact convex set under a continuous map, must be a closed interval.  

The proof of \eqref{LD estimate} will be given in Appendix \ref{section Appendix}, see Lemma \ref{lemma LD}.
\end{proof}

{\color{red} Now, we can show the main application of this paper: we determine the cutout dimension of the Poisson cutout set $E$ in the context of the self-conformal set $X$. }

\begin{theorem}\label{thm:main}
{\color{blue}Suppose that $E$ is the Poissonian cutout set on $X$, where the intensity is $\gamma\mathcal{H}\times\rho$, and $\mathcal{H}$ is the natural self-conformal measure on $X$.}
If {\color{blue}$m=m(\gamma)=\max_{\alpha_{min}\le \alpha\le \alpha_0}f(\alpha)-\gamma \alpha\ge 0$, } then almost surely $\dim_H(E)\le m$ and $\dim_H(E)\ge m$ with positive probability. If $m<0$, then $E=\varnothing$ almost surely.
\end{theorem}

\begin{proof}
This is a consequence of \eqref{continuity spectrum} and  \eqref{LD estimate} of  Proposition \ref{Prop multifractal properties} and Theorem \ref{main theorem}.
\end{proof}

\medskip

\begin{ex}

{\color{red}Let $X\subset \R^d$ be  a self-similar set satisfying the strong separation condition.  Suppose that the maps $\{g_i\}_{i=1}^\ell$ have equal contraction ratios} (For instance, $X$ could be the classical ternary Cantor set), that is, there is constant $0<a<1$ such that $|g_i'|=a$ for all $i,j\in \Lambda$.  Then, in this case, $|DS|$ is constant on $X$ and $F$ is an additive sequence (see \cite[Chapter 6.2]{Falconer97}). Moreover $F$ is H\"older continuous. It is well known that (see e.g. \cite{Fan1994,RuelleBook,PesinWeiss97} )
 the multifractal spectrum $f(\alpha)$ is analytical, strictly convex on $\Omega$ and for any $\alpha\in\Omega$ we have
\begin{equation}\label{pressure spectrum Legendre}
f(\alpha)=\inf_{q\in \R}\left(\tilde{P}(q)-\alpha q\right)
\end{equation}
where $\tilde{P}(q)=\frac{P(qF)}{-\log a}$. We make two remarks:
\begin{enumerate}
\item Observe that since $f'(\alpha_0)=0$, we have $m(\gamma)>Q-\gamma \alpha_0$. Thus, the almost sure dimension of $E$ is not due to the $\mathcal H$-almost sure value of $\ad(\mathcal H,x)$ but is affected by the multifractal behaviour of the average densities. 
\item From \eqref{pressure spectrum Legendre}, one can show that $m(\gamma)=\tilde{P}(-\gamma)=\frac{P(-\gamma F)}{-\log a}$. This means that the critical value (regarding the parameter $\gamma$) for the emptiness (or for the positivity of the Hausdorff dimension) of $E$ is the unique zero of the pressure function (the pressure function in our case is strictly monotone).
\end{enumerate}

\end{ex}

\begin{remark}\label{rems:final}




 It seems plausible that in Theorem \ref{thm:main}, $\dim_H(E)$ is equal to the cutout dimension almost surely conditioned on $E\neq\varnothing$. In other words, $\mathbb{P}(E\neq\varnothing\text{ and }\dim_H (E)<m)=0$. However, the proof only implies $\dim_H(E)=m$ almost surely on $\nu(X)>0$, where $\nu$ is the random measure as in Lemma \ref{lemma:lb} corresponding to the value of $\alpha$ so that $m=f(\alpha)-\gamma\alpha$. We expect that $\mathbb{P}(\nu(X)=0\text{ and }E\neq\varnothing)=0$, but haven't been able to prove this. As pointed out in \cite{ShmerkinSuomala}, this problem is open also in the case of $X=[0,1]$, $\mathcal{H}=\mathcal{L}$.
\end{remark}


\appendix
\section{}\label{section Appendix}

In this Appendix, we give the proof of the following lemma which is 
the statement \eqref{LD estimate} of Proposition \ref{Prop multifractal properties}.

\begin{lemma} \label{lemma LD}
Under the setting of Proposition \ref{Prop multifractal properties}, we have
{\color{cyan}
$$ \mathcal H (\Delta_r(0,\beta))\le C r^{Q-f(\beta)-\varepsilon}$$ for all $r>0$ whenever $0<\beta\le\alpha_0$ and all $\varepsilon>0$. Here $C$ is a finite constant that is allowed to depend on $\beta$ and $\varepsilon$ (but not on $r$!).}
\end{lemma}

{\bf Notations and classical estimates.}
For $u\in \Lambda^*$, let $\tilde{u}$ be the word obtained by erasing the last letter. For $0<\tau<1$, consider the ``cut-set"
$$W_\tau=\{u\in\Lambda^* : \ {\rm diam}(g_u(X))\leq \tau \ {\rm and}\  {\rm diam}(g_{\tilde{u}}(X))>\tau\}.$$

It is clear that for any $0<\tau<1$, $\Lambda^\infty=\bigsqcup_{u\in W_\tau}[u]$ and the IFS $\{g_u\}_{u\in W_\tau}$ generates the same attractor $X$, moreover $\mathcal{H}$ is the natural measure associated to $\{g_u\}_{u\in W_\tau}$. For any $x\in X$, there exists $(v_n)_{n\geq1}\in W_\tau^\infty$ such that $\{x\}=\lim_n g_{v_1^n}(X)=:g_{v_1^\infty}(X)$. We denote the inverse map corresponding to the IFS $\{g_u\}_{u\in W_\tau}$ by $S_\tau$, so that we have {\color{cyan} $\{S_{\tau}(x)\}=g_{v_2^\infty}(X)$ and more generally $\{S_\tau^n(x)\}=g_{v_{n+1}^\infty}(X)$.}

A well known calculation (see e.g. \cite{Patzschke97}) shows that a $C^{1+\varepsilon}$ conformal iterated function system satisfies the bounded distortion principle: there exists $L>1$ such that 
$$L^{-1}\leq \frac{\|g'_u(x)\|}{\|g'_u(y)\|}\leq L\  \textrm{ for all }\  u\in \Lambda^*, x,y\in X. $$
Let $\lambda_0=\min \{||g'_i(x)||:i\leq \ell,x\in X\}>0$. {\color{cyan} Recall that $\ell$ is the number of maps in the IFS $G=\{g_i\}_{i=1}^\ell$.} Then for any $u=u_1\cdots u_n\in \Lambda^*$ and $y\in X$,
$$\|g'_u(y)\|=\|g'_{\tilde{u}}(g_{u_n}(y))\|\|g'_{u_n}(y)\|\geq L^{-1}\lambda_0\max_{z\in X}\|g'_{\tilde{u}}(z)\|.$$
Now let $u\in W_\tau$. Then 
$$\tau\leq {\rm diam}(g_{\tilde{u}}(X))\leq \max_{z\in X}\|g'_{\tilde{u}}(z)\|{\rm diam}(X)\leq L\lambda_0^{-1}{\rm diam}(X)\min_{z\in X}\|g'_u(z)\|.$$
On the other hand, $X=g^{-1}_u(g_u(X))$ so we have 
$${\rm diam}(X)\leq \max_{z\in X}\|(g'_u)^{-1}(z)\|{\rm diam}(g_u(X))\leq \max_{z\in X}\|(g'_u)^{-1}(z)\|\cdot \tau$$
and 
$$\max_{z\in X}\|g'_u(z)\|=\left(\min_{z\in X}\|(g'_u)^{-1}(z)\|\right)^{-1}\leq L\tau{\rm diam}(X)^{-1}.$$
So there exists a constant $C>1$ such that for any $0<\tau<1$ and any $u\in W_\tau$, we have 
\begin{equation}\label{equation derivative cut-set}
\tau C^{-1}\leq \min_{z\in X}\|g'_u(z)\|, \ \max_{y\in X}\|g'_u(y)\|\leq \tau C.
\end{equation}
From \eqref{equation derivative cut-set}, we deduce that 
\begin{equation}\label{equation derivative inverse cut-set}
\tau^{-n}C^{-n}\leq \min_{z\in X}\|DS_\tau^n(z)\|, \ \max_{y\in X}\|DS_\tau^n(y)\|\leq \tau^{-n}C^{n}.
\end{equation}

Now we can give the proof of Lemma \ref{lemma LD}.

\begin{proof}[Proof of Lemma \ref{lemma LD}]
{\color{red} We first give the proof for the case $\alpha_{\min}\leq\beta\le\alpha_0$}. Fix $x\in X$ and a small $0<r<1$. Let $n=n(x,r)\in \N$ be such that 
$$|DS^{n+1}_\tau(x)|^{-1}\leq r\leq  |DS^{n}_\tau(x)|^{-1}.$$
From \eqref{equation derivative inverse cut-set}, we know that $\frac{\log r}{\log \tau-\log C}\leq n\leq \frac{\log r}{\log \tau+\log C}$. Here and in the rest of the proof, we always take a $\tau<C^{-1}$ so that $\log \tau+\log C<0$. Then we have 
$$\frac{\int_{r}^1\mathcal{H}(B(x,t))t^{-Q-1}\,dt}{-\log r}\geq \frac{\int_{ |DS^{n}_\tau(x)|^{-1}}^1\mathcal{H}(B(x,t))t^{-Q-1}\,dr}{\log  |DS^{n+1}_\tau(x)|}\,.$$
So we get 
$$\left\{x\in X : \ad(\mathcal{H},x,r)\leq \beta\right\}\subset \left\{x\in X :  \frac{\int_{ |DS^{n}_\tau(x)|^{-1}}^1\mathcal{H}(B(x,t))t^{-Q-1}\,dr}{\log  |DS^{n+1}_\tau(x)|}\leq \beta \right\}=:A_{\tau,n}.$$
Thus we have
$$\frac{\log\mathcal{H}(\Delta_r(0,\beta))}{-\log r}\leq \frac{\log \mathcal{H}(A_{\tau,n})}{-\log r}\leq \frac{\log \mathcal{H}(A_{\tau,n})}{\log |DS^{n+1}_\tau(x)|}\leq \frac{\log \mathcal{H}(A_{\tau,n})}{(n+1)(-\log \tau+\log C)}.$$
For proving the claim of the lemma we only need to show that
$$\limsup_{\tau\to0^+}\limsup_{n\to\infty}\frac{\log \mathcal{H}(A_{\tau,n})}{-n\log \tau}\leq Q-f(\beta).$$

Recall that we can rewrite $A_{\tau,n}$ as 
$$A_{\tau,n}=\left\{x\in X :  \frac{\sum_{k=0}^{n-1}f_k^\tau(x)}{\log  |DS^{n+1}_\tau(x)|}\leq \beta \right\}$$
where $f_k^\tau(x)=\int_{ |DS^{k+1}_\tau(x)|^{-1}}^{ |DS^{k}_\tau(x)|^{-1}}\mathcal{H}(B(x,t))t^{-Q-1}\,dr$ which is asymptotically additive for the system $(X,S_\tau)$.
By Chebyshev's inequality, for any $\lambda\geq 0$
\begin{align*}
\mathcal{H}(A_{\tau,n})
\le&\int_X \exp\left(\lambda \left (\beta \log  |DS^{n+1}_\tau(x)|-\sum_{k=0}^{n-1}f_k^\tau(x)\right)\right)d \mathcal{H}(x)\\
\le& \sum_{v_1^n\in W_\tau^n} \mathcal{H}(g_{v_1^n}(X))\sup_{x\in g_{v_1^n}(X)}\exp\left(\lambda \left (\beta \log  |DS^{n+1}_\tau(x)|-\sum_{k=0}^{n-1}f_k^\tau(x)\right)\right).
\end{align*}

Since $\mathcal{H}$ is the natural measure of the IFS $(g_u)_{u\in W_\tau}$, we have that
$\mathcal{H}(g_{v_1^n}(X))\asymp\exp(-Q\log |DS_\tau^n(x)|)$ for any $x\in g_{v_1^n}(X)$. Whence
\begin{equation}\label{equ app lemma 1}
\begin{split}
&\mathcal{H}(A_{\tau,n})
\lesssim\\  &\sum_{v_1^n\in W_\tau^n}  \sup_{x\in g_{v_1^n}(X)}\exp\left(\lambda \left (\beta \log  |DS^{n+1}_\tau(x)|-\sum_{k=0}^{n-1}f_k^\tau(x)\right) -Q\log |DS_\tau^n(x)|\right)\,, 
\end{split}
\end{equation}
whenever $\lambda\geq 0$.
Note that $|\log  |DS^{n+1}_\tau(x)|-\log  |DS^{n}_\tau(x)||\leq \max_{z\in X}\log|DS_\tau(z)|$. {\color{blue} Here the notation $A\asymp B$ means that $C^{-1}A\le B\le C A$ and $A\lesssim B$ stands for $A\le C B$ for a constant $C<\infty$ independent of $n, \tau$ and $v_1^{n}$. } Taking logarithms and dividing both sides of \eqref{equ app lemma 1} by $n$  and then taking limsup we get 

\begin{equation}\label{equ app lemma 1-1}
\limsup_n\frac{\log \mathcal{H}(A_{\tau,n})}{n}\le  P_\tau(\lambda(\beta \log DS_\tau-F_\tau)-Q\log DS_\tau),\quad\quad \lambda\geq0. 
\end{equation}
where $P_\tau(\lambda(\beta \log DS_\tau-F_\tau)-Q\log DS_\tau)$ is the pressure function (of the system $(X,S_\tau)$) associated to the sequence of functions 
$$\left\{\lambda \left (\beta \log  |DS^{n}_\tau(x)|-\sum_{k=0}^{n-1}f_k^\tau(x)\right) -Q\log |DS_\tau^n(x)|\right\}\,.$$

We now show that the inequality \eqref{equ app lemma 1-1} holds also when $\lambda<0$.
For this, we only need to show that $P_\tau(\lambda(\beta \log DS_\tau-F_\tau)-Q\log DS_\tau)\ge 0$ for $\lambda<0$.  Fix $\lambda<0$. Denote $B_n=\int_X \exp\left(\lambda \left (\beta \log  |DS^{n}_\tau(x)|-\sum_{k=0}^{n-1}f_k^\tau(x)\right)\right)d \mathcal{H}(x)$. We are going to prove that $\limsup_n\frac{\log B_n}{n}\geq 0$, which will imply $P_\tau(\lambda(\beta \log DS_\tau-F_\tau)-Q\log DS_\tau)\ge 0$. 

By Jensen's inequality we have
$$B_n\geq \exp\left( \int_X \lambda \left (\beta \log  |DS^{n}_\tau(x)|-\sum_{k=0}^{n-1}f_k^\tau(x)\right)d \mathcal{H}(x)\right).$$
So 
\begin{equation}\label{equ app lemma 2}
\frac{\log B_n}{n}\ge  \int_X \lambda \left (\beta \frac{\log  |DS^{n}_\tau(x)|}{n}-\frac1n\sum_{k=0}^{n-1}f_k^\tau(x)\right)d \mathcal{H}(x) .
\end{equation}
We know that
$$\lim_{n\to\infty}\frac{\sum_{k=0}^{n-1}f_k^\tau(x)}{\log  |DS^{n}_\tau(x)|}=\alpha_0, \text{ for } \mathcal{H}\textrm{-a.e.} \ x$$
Since $\alpha_0\ge\beta$ and $\lambda<0$, in view of \eqref{equ app lemma 2}, we get
$$\limsup_n\frac{\log B_n}{n}\geq 0.$$
So we have proved that
$$\limsup_n\frac{\log \mathcal{H}(A_{\tau,n})}{n}\le  \inf_{\lambda\in \R} P_\tau(\lambda(\beta \log DS_\tau-F_\tau)-Q\log DS_\tau). $$

For completing the proof, we only need to show that {\color{red}for $\beta\in [\alpha_{\min},\alpha_0]$},
\begin{equation}\label{equ app lemma 3}
\limsup_{\tau\to0^+}\frac{ \inf_{\lambda\in \R} P_\tau(\lambda(\beta \log DS_\tau-F_\tau)-Q\log DS_\tau)}{-\log \tau}\leq Q-f(\beta).
\end{equation}
From the definition of the pressure function $P_\tau$ and the fact $n(-\log\tau-\log C)\le \log|DS^n_\tau(x)|\le n(-\log\tau+\log C)$, we deduce that 
\begin{align*}
&|P_\tau\left(\lambda(\beta \log DS_\tau-F_\tau)-Q\log DS_\tau\right)\\
&-P_\tau\left(\lambda(\beta \log DS_\tau-F_\tau\right)-f(\beta)\log DS_\tau)-(Q-f(\beta))\log\tau|\le 2\log C\,.
\end{align*}
So for proving \eqref{equ app lemma 3}, it is sufficient to show that {\color{red}for $\beta\in [\alpha_{\min},\alpha_0]$},
$$\inf_{\lambda\in \R}P_\tau(\lambda(\beta \log DS_\tau-F_\tau)-f(\beta)\log DS_\tau)=0,$$ 
but this is exactly the statement \eqref{multifractal pressure zero point} of Proposition \ref{Prop multifractal properties} for the system $(X,S_\tau)$. 

{\color{red} Now, we assume that $\beta<\alpha_{\min}$.} We will show that $\Delta_r(0,\beta)=\emptyset$ when $r$ is small enough, this clearly implies the desired result. To this end, we only need to prove that 
\begin{equation}\label{eq:app:lemma:1}
\underline{\alpha}:=\lim_{n\to\infty}\inf_{x\in X}\frac{\sum_{k=0}^{n-1}f_k(x)}{\log DS^n(x)}=\alpha_{\min}.
\end{equation}
By Proposition \ref{Prop multifractal properties}, $\Delta(\alpha_{\min})\neq \emptyset$, i.e., there exists $x\in X$ such that $A(\mathcal{H},x)=\alpha_{\min}$, so $\underline{\alpha}\leq\alpha_{\min}$. It remains to show the reverse inequality. Pick $n_i\to\infty$ and $x_i\in X$ such that 
$\lim_{i\to\infty} \frac{\sum_{k=0}^{n_i-1}f_k(x_i)}{\log DS^{n_i}(x_i)}=\underline{\alpha}$. Let $\mu_{n_i}=\frac{1}{n_i}\sum_{j=0}^{n_i-1}\delta_{S^jx_i}$ for $i\in \N$. Up to taking a subsequence of $(n_i)_i$, we can suppose that $\mu_{n_i}$ converges weakly to some probability measure $\mu$. By \cite[Lemma A.4. (ii)]{FengHuang10}, we know that $\mu \in \mathcal{M}(X,S)$, and moreover
$$\lim_{i\to\infty} \frac{1}{n_i}\int_X \sum_{k=0}^{n_i-1}f_k(x)d\delta_{x_i}=F_*(\mu) \ {\rm and}\ \lim_{i\to\infty} \frac{1}{n_i}\int_X \log DS^{n_i}(x)d\delta_{x_i}=\log DS_*(\mu).$$
Thus we have
$$\underline{\alpha}=\frac{F_*(\mu)}{\log DS_*(\mu)}\geq \alpha_{\min}.$$
This ends the proof of the lemma.
\end{proof}


\bibliographystyle{plain}
\bibliography{bibliography}

\end{document}